\documentclass[preprint,12pt]{elsarticle}
\usepackage{latexsym}
\usepackage{amssymb}
\usepackage{amsmath}
\usepackage{amsthm}
\usepackage{verbatim}
 \usepackage[pagewise]{lineno}
\numberwithin{equation}{section}
\newtheorem{theorem}{Theorem}[section]
\newtheorem{lemma}{Lemma}[section]

\newtheorem{definition}{Definition}[section]
\newcommand{\sign}[1]{\mathrm{sgn}(#1)}

\allowdisplaybreaks
\usepackage{geometry}
\biboptions{sort&compress}

\begin{document}
\begin{frontmatter}


\title{Globally conservative weak solutions for a class of two-component nonlinear
dispersive wave equations beyond wave breaking
}

\author[ad1]{Yonghui Zhou}
\ead{zhouyhmath@163.com}
\author[ad3]{Xiaowan Li\corref{cor}}
\ead{xiaowan0207@163.com}
\address[ad1]{School of Mathematics, Hexi University, Zhangye 734000, P.R. China}
\address[ad3]{College of Mathematics and System Sciences, Xinjiang University, Urumqi, 830046, P.R. China}
\cortext[cor]{Corresponding author.}
\begin{abstract}

In this paper, we prove that the existence of globally conservative weak solutions for a class of two-component nonlinear dispersive wave equations beyond wave breaking. We first introduce a new set of independent and dependent variables in connection with smooth solutions, and transform the system into an equivalent semi-linear system.
We then establish the global existence of solutions for the semi-linear system via
the standard theory of ordinary differential equations. Finally, by the inverse transformation method, we prove the existence of the globally conservative weak solution for the original system.

\end{abstract}

\end{frontmatter}


\section{Introduction}
\label{sec:1}

In this paper, we consider the Cauchy problem for a class of two-component nonlinear
dispersive wave equations with the following form
\begin{equation}
\begin{cases}
u_{t}-u_{txx}+(f(u))_{x}-(f(u))_{xxx}+\left(g(u)+\frac{f''(u)}{2}u_{x}^{2}\right)_{x}+\rho \rho_{x}=0,\\
\rho_{t}+f'(u)\rho_{x}+\left(\frac{1}{2}+\frac{f''(u)}{2}\right)\rho u_{x}=0,
\label{101}
\end{cases}
\end{equation}
subject to the initial data
\begin{equation}
u(0,x)=u_{0}(x),\ \rho(0,x)=\rho_{0}(x),
\label{102}
\end{equation}
which was studied by Novruzov and Bayrak \cite{Novruzov2022}, where $u(t,x)$ denotes the horizontal velocity of the fluid and $\rho(t,x)$ is a parameter related to the free surface elevation from equilibrium, $f(u), g(u)\in C^{\infty}(\mathbb{R},\mathbb{R})$ satisfy $g(0)=0$.

For $\rho\equiv 0$, system \eqref{101} is cast to the following nonlinear
dispersive wave equations
\begin{equation}
u_{t}-u_{txx}+(f(u))_{x}-(f(u))_{xxx}+\left(g(u)+\frac{f''(u)}{2}u_{x}^{2}\right)_{x}=0,
\label{103}
\end{equation}
which was proposed by Holden and Raynaud \cite{Holden2007}.
 Especially, if $f(u)=-7u^{2}-u$ and $g(u)=2u+10u^{2}-2u^{3}+3u^{4}$, Eq. \eqref{103} is cast to the following Constantin-Lannes equation
\begin{equation}
u_{t}-u_{txx}+u_{x}+6uu_{x}-6u^{2}u_{x}+12u^{3}u_{x}+u_{xxx}=-28u_{x}u_{xx}-14uu_{xxx},
\label{103a}
\end{equation}
which is proposed by Constantin and Lannes \cite{Constantin2009} as an asymptotic model to the Green-Naghdi equations.
If $f(u)=\frac{k}{2}u^{2}$, Eq. \eqref{103} is cast to the following generalized hyper-elastic rod wave equation
\begin{equation}
u_{t}-u_{txx}+\left(g(u)\right)_{x}=k\left(2u_{x}u_{xx}+uu_{xxx}\right),
\label{104}
\end{equation}
which was firstly studied by Coclite, Holden and Karlsen \cite{Coclite20051,Coclite20052}. If $g(u)=\frac{3}{2}u^{2}$, Eq. \eqref{104} reads as the so-called hyper-elastic rod wave equation
\begin{equation}
u_{t}-u_{txx}+3uu_{x}=k\left(2u_{x}u_{xx}+uu_{xxx}\right),
\label{105}
\end{equation}
which was proposed by Dai \cite{Dai1,Dai2}, and describes far-field, finite length, finite amplitude radial deformation waves in cylindrical compressible hyper-elastic rods and $u(t,x)$ represents the radial stretch relative to a pre-stressed state.
If $k=1$, Eq. \eqref{105} is reduced to the classical Camassa-Holm(CH) equation
\begin{equation}
u_{t}-u_{txx}+3uu_{x}=2u_{x}u_{xx}+uu_{xxx},
\label{106}
\end{equation}
which is a nonlinear dispersive wave equation that models the propagation of unidirectional irrotational shallow water waves over a flat bed. Such model was theoretically implied in the discussion of symmetries, symplectic and bi-Hamiltonian structure \cite{Fokas1981}, but was physically derived from Euler's water wave equations with addressing the significant peakon solution by Camassa and Holm \cite{Camassa1993} in 1993. 
A very interesting physical feature of CH equation is that it models wave breaking phenomena 
\cite{Constantin19971,Constantin1,Constantin19981,Constantin1998,Constantin19983,Constantin2000}. The CH equation also admits peaked traveling waves that interact like solitons \cite{Beals2000,Camassa1993,Lenells2005} and, moreover, these peaked waves are orbitally stable \cite{Constantin20003,Lenells2004}.

For $f(u)=\frac{u^{2}}{2}$ and $g(u)=ku+u^{2}$, system \eqref{101} becomes the two-component model for CH equation
\begin{equation}
\begin{cases}
u_{t}-u_{txx}+ku_{x}+3uu_{x}+2u_{x}u_{xx}+uu_{xxx}+\rho \rho_{x}=0,\\
\rho_{t}+u\rho_{x}+\rho u_{x}=0,
\label{107}
\end{cases}
\end{equation}
which was derived first in \cite{Shabat2002}, where $k$ is a dispersive coefficient related to the critical shallow water speed. In the paper \cite{Hone2017} a classification of integrable two-component systems of non-evolutionary partial differential equations that are analogous to the CH equation is carried out via the perturbative symmetry approach. Notice that the CH equation can be obtained via the obvious reduction $k=0$ and $\rho=0$.

In the last 30 years, the CH equation and its various generalizations
were tremendously investigated due to their many very interesting properties and remarkable solutions,
see \cite{Camassa1993,Constantin2001,Constantin19971,Constantin1,Constantin19981,
Li2000,Constantin1998,Constantin19982,Constantin19983,Constantin2000,
Constantin19984,Constantin20004,Wahlen2006,Dai1,Dai2,HolmStaley2003,Qiao2003CMP,Hone2017,
Qiao2006JMP,QiaoZhang2006,Coclite20051,Bressan2007,Bressan20071,Constantin2009,
Qiao2007JMP,Constantin20003,Chen2011,Brandolese2014,Brandolese20141,Ji2021,Novruzov2017,
Ji2022,Novruzov2022,Coclite20052,Zhou2022,Brandolese20142} and the reference therein. However, the related works are mainly to consider the behavior of strong solutions before or during the occurrence of wave breaking phenomena. In view of the possible development of singularities in finite time for the strong solutions, another interesting issue is concerned with the behavior of solutions beyond the occurrence of wave breaking phenomena. In 2007, by the characteristic method, Bressan and Constantin firstly proved that solutions of the CH equation can be continued as either globally conservative weak solution \cite{Bressan2007} or globally dissipative weak solutions \cite{Bressan20071}. Following the idea of \cite{Bressan2007}, Mustafa \cite{Mustafa20072} obtained the globally conservative weak solutions of Eq. \eqref{104};
Zhou, Ji and Qiao \cite{Zhou2024} investigated the existence of globally conservative weak solutions for Eq. \eqref{103}.

Inspired by the previous work, in this paper, we study the globally conservative weak solutions of the Cauchy problem \eqref{101}--\eqref{102}, but beyond wave breaking.

The rest of this paper is organized as follows. In Section \ref{sec:2}, we give the energy conservation law and some basic estimates. In Section \ref{sec:3}, we introduce a new set of independent and dependent variables in connection with smooth solutions, and transform system \eqref{101} into an equivalent semi-linear system. In Section \ref{sec:4}, we establish the global existence of solutions for the semi-linear system. In Section \ref{sec:5}, by the inverse transformation method, we prove the existence of the globally conservative weak solution for system \eqref{101}.

\textbf{Notation.} Throughout this paper, all spaces of functions are over $\mathbb{R}$, and for simplicity, we drop $\mathbb{R}$ in our notation of function spaces if there is no confusion. 
Additionally, we denote by $\|\cdot\|_{s}$ the norm in the Sobolev space $H^{s}(\mathbb{R})$, and denote by $\|\cdot\|_{L^{p}}$ the norm in the Lebesgue space $L^{p}(\mathbb{R})$. $f'(u)$ and $g'(u)$ represent the derivatives of $f(u)$ and $g(u)$ with respect to $u$, respectively.

\section{Preliminary}

\setcounter{equation}{0}

\label{sec:2}

In this section, we consider the following nonlocal form of a class of nonlinear
dispersive wave equations
\begin{equation}
u_{t}+f'(u)u_{x}+P_{x}=0,
\label{201}
\end{equation}
\begin{equation}
\rho_{t}+f'(u)\rho_{x}=-\left(\frac{1}{2}+\frac{f''(u)}{2}\right)\rho u_{x},
\label{201a}
\end{equation}
equivalent to system \eqref{101}, where the source term $P$ is defined by
\begin{equation}
P=p\ast\left(g(u)+\frac{f''(u)}{2}u_{x}^{2}+\frac{\rho^{2}}{2}\right)
=\frac{1}{2}e^{-|x|}\ast\left(g(u)+\frac{f''(u)}{2}u_{x}^{2}+\frac{\rho^{2}}{2}\right).
\label{202}
\end{equation}

For smooth solutions, differentiating \eqref{201} with respect to $x$ and using the relation $p_{xx}\ast h=p\ast h-h$, we get
\begin{equation}
u_{tx}+\frac{f''(u)}{2}u_{x}^{2}+f'(u)u_{xx}+P-g(u)-\frac{\rho^{2}}{2}=0.
\label{203}
\end{equation}
Multiplying \eqref{201} by $2u$, \eqref{201a} by $2\rho$ and \eqref{203} by $2u_{x}$, and adding the three resulting equations, we can get the following equation
\begin{equation}
(u^{2}+u_{x}^{2}+\rho^{2})_{t}+\left(f'(u)(u^{2}+u_{x}^{2}+\rho^{2})\right)_{x}
=-2(uP)_{x}+\left(2g(u)+f''(u)u^{2}\right)u_{x}.
\label{204}
\end{equation}
Define
\begin{equation}
H(u)=\int_{0}^{u}\left(2g(s)+f''(s)s^{2}\right)ds.
\label{205}
\end{equation}
Then \eqref{204} can be rewritten as
\begin{equation}
(u^{2}+u_{x}^{2}+\rho^{2})_{t}+\left(f'(u)(u^{2}+u_{x}^{2}+\rho^{2})\right)_{x}
=\left(H(u)-2uP\right)_{x}.
\label{206}
\end{equation}
Integrating \eqref{206} with respect to $t$ and $x$ over $[0,t]\times \mathbb{R}$, we get
\begin{equation}
E(t)=\int_{\mathbb{R}}(u^{2}(t,x)+u_{x}^{2}(t,x)+\rho^{2})dx=E(0).
\label{207}
\end{equation}

Since $f(u), g(u)\in C^{\infty}(\mathbb{R},\mathbb{R})$ and $g(0)=0$, then we have
\begin{equation}
|g(u(x))|\leq \sup_{|s|\leq\|u\|_{L^{\infty}}}|g'(s)||u(x)|\leq C(\|u\|_{1})|u(x)|.
\label{208}
\end{equation}
On the other hand, we can also easily verify that $|f''(u)|<C_{1}$ for some positive constant $C_{1}$. Therefore, we get
\begin{align}
\|P\|_{L^{2}}
\leq&\frac{1}{2}\left\|e^{-|x|}\right\|_{L^{1}}\|g(u)\|_{L^{2}}+ \frac{1}{2}\left\|e^{-|x|}\right\|_{L^{2}}\left\|\frac{f''(u)}{2}u_{x}^{2}+\frac{\rho^{2}}{2}\right\|_{L^{1}}\nonumber\\
\leq& C\left(\|g(u)\|_{L^{2}} +\|u\|_{1}^{2}+\|\rho\|_{L^{2}}\right)\nonumber\\
\leq& CE(0).
\label{209}
\end{align}
Similarly, we can obtain
\begin{equation}
\|P_{x}\|_{L^{2}}, \|P\|_{L^{\infty}},\|P_{x}\|_{L^{\infty}}\leq CE(0).
\label{2010}
\end{equation}

\section{Semi-linear system for smooth solutions}

\setcounter{equation}{0}

\label{sec:3}

In this section, we establish a semi-linear system for smooth solutions. To this end, it is essential to introduce the characteristic equation
\begin{equation}
\frac{dx(t)}{dt}=f'(u(t,x)).
\label{301}
\end{equation}
For any fixed point, the characteristic curve crossing the point $(t,x)$ is defined by setting
\begin{equation}
\gamma\rightarrow x^{c}(\gamma;t,x).
\label{302}
\end{equation}

In what follows, we use the energy density $(1+u_{x}^{2}(0,\bar{x}))$ to define the characteristic coordinate $Z=Z(t,x)$,
\begin{equation}
Z(t,x):=\int_{0}^{x^{c}(0;t,x)}(1+u_{x}^{2}(0,\bar{x}))d\bar{x}.
\label{303}
\end{equation}
Therefore, $Z(t,x)$ satisfies
\begin{equation}
Z_{t}+f'(u)Z_{x}=0, \ \ (t,x)\in \mathbb{R_{+}}\times\mathbb{R}.
\label{304}
\end{equation}
We also define $T=t$ to obtain the new coordinate $(T,Z)$. Then for any smooth function $h(T,Z)=h(t,Z(t,x))$, by \eqref{304}, we get
\begin{align}
h_{t}+f'(u)h_{x}
&=h_{T}T_{t}+h_{Z}Z_{t}+f'(u)\left(h_{T}T_{x}+h_{Z}Z_{x}\right)\nonumber\\
&=h_{T}\left(T_{t}+f'(u)T_{x}\right)+h_{Z}\left(Z_{t}+f'(u)Z_{x}\right)\nonumber\\
&=h_{T}
\label{305}
\end{align}
and
\begin{equation}
h_{x}=h_{T}T_{x}+h_{z}Z_{x}=h_{z}Z_{x}.
\label{306}
\end{equation}
Furthermore, we denote
\begin{eqnarray*}
u(T,Z):=u(T,x(T,Z)),\ \rho(T,Z):=\rho(T,x(T,Z)),\\
P(T,Z):=P(T,x(T,Z))\ \ \text{and}\ \ P_{x}(T,Z):=P_{x}(T,x(T,Z)).
\end{eqnarray*}

In what follows, we define
\begin{equation}
w:=2\arctan u_{x}\ \ \text{and}\ \ v:=(1+u_{x}^{2})\frac{\partial x}{\partial Z}
\label{307}
\end{equation}
with $u_{x}=u_{x}(T,x(T,Z))$. By \eqref{307}, we can easily verify that
\begin{equation}
\frac{1}{1+u_{x}^{2}}=\cos^{2}\frac{w}{2},\ \ \frac{u_{x}^{2}}{1+u_{x}^{2}}=\sin^{2}\frac{w}{2},\ \
\frac{u_{x}}{1+u_{x}^{2}}=\frac{1}{2}\sin w,
\label{308}
\end{equation}
\begin{equation}
\frac{\partial x}{\partial Z}=\frac{v}{1+u_{x}^{2}}=v\cos^{2} \frac{w}{2}.
\label{309}
\end{equation}
By \eqref{309}, for any time $t=T$, we have
\begin{equation}
x(T,Z')-x(T,Z)=\int_{Z}^{Z'}\left(v\cos^{2} \frac{w}{2}\right)(T,s)ds.
\label{3010}
\end{equation}

Let $y=x(T,Z')$ and $x=x(T,Z)$. By using the identities \eqref{308}--\eqref{3010}, we get
\begin{align}
P(Z)
=&P(T,Z)\nonumber\\
=&\frac{1}{2}\int_{\mathbb{R}}e^{-|x(T,Z)-y|}\left(g(u)
+\frac{f''(u)}{2}u_{x}^{2}+\frac{\rho^{2}}{2}
\right)(T,y)dy\nonumber\\
=&\frac{1}{2}\int_{\mathbb{R}}e^{-|\int_{Z}^{Z'}\left(v\cos^{2} \frac{w}{2}\right)(T,s)ds|}\Big(g(u(Z'))\cos^{2} \frac{w(Z')}{2}+\frac{f''(u(Z'))}{2}\sin^{2} \frac{w(Z')}{2}\nonumber\\
&+\frac{\rho^{2}}{2}\cos^{2} \frac{w(Z')}{2}\Big)v(Z')dZ'
\label{3011}
\end{align}
and
\begin{align}
P_{x}(Z)
=&P_{x}(T,Z)\nonumber\\
=&\frac{1}{2}\left(\int_{x(T,Z)}^{+\infty}
-\int_{-\infty}^{x(T,Z)}\right)e^{-|x(T,Z)-y|}\left(g(u)
+\frac{f''(u)}{2}u_{x}^{2}+\frac{\rho^{2}}{2}
\right)(T,y)dy\nonumber\\
=&\frac{1}{2}\left(\int_{Z}^{+\infty}-\int_{-\infty}^{Z}\right)
e^{-|\int_{Z}^{Z'}\left(v\cos^{2} \frac{w}{2}\right)(T,s)ds|}\bigg(g(u(Z'))\cos^{2} \frac{w(Z')}{2}\nonumber\\
&+\frac{f''(u(Z'))}{2}\sin^{2} \frac{w(Z')}{2}+\frac{\rho^{2}}{2}\cos^{2} \frac{w(Z')}{2}\bigg)v(Z')dZ'.
\label{3012}
\end{align}

In what follows, we derive a closed semi-linear system for the unknowns $u, \rho, w$ and $v$ under the new variables $(T,Z)$. From \eqref{201} and \eqref{301}, we get
\begin{equation}
u_{T}(T,Z)=u_{t}(T,Z)+f'(u)u_{x}(T,Z)=-P_{x}(T,Z),
\label{3013}
\end{equation}
where $P_{x}(T,Z)$ is given at \eqref{3012}.
From \eqref{203} and \eqref{307}, we get
\begin{align}
w_{T}(T,Z)=&\frac{2}{1+u_{x}^{2}}\left(u_{tx}+f'(u)u_{xx}\right)(T,Z)\nonumber\\
=&\frac{2}{1+u_{x}^{2}}\left(-\frac{f''(u)}{2}u_{x}^{2}+g(u)-P+\frac{\rho^{2}}{2}\right)\nonumber\\
=&2\left(g(u)-P+\frac{\rho^{2}}{2}\right)\cos^{2}\frac{w}{2}-f''(u)\sin^{2}\frac{w}{2}
\label{3014}
\end{align}
and
\begin{align}
\rho_{T}(T,Z)=\rho_{t}+f'(u)\rho_{x}
=&-\left(\frac{1}{2}+\frac{f''(u)}{2}\right)\rho u_{x}\nonumber\\
=&-\left(\frac{1}{2}+\frac{f''(u)}{2}\right)\rho \tan \frac{w}{2},
\label{3014a}
\end{align}
where $P=P(T,Z)$ is given at \eqref{3011}.

Below, we will derive the equation for $v(T,Z)$. To this end, we need to use the following relation
\begin{equation}
Z_{tx}+f'(u)Z_{xx}=-f''(u)u_{x}Z_{x},
\label{3015}
\end{equation}
which can be derived from \eqref{304}. Then \eqref{203}, \eqref{305}, \eqref{307} and \eqref{3015} yield
\begin{align}
v_{T}(T,Z)=\left(\frac{1+u_{x}^{2}}{Z_{x}}\right)_{T}
=&\frac{Z_{x}(1+u_{x}^{2})_{T}-(1+u_{x}^{2})Z_{xT}}{Z^{2}_{x}}\nonumber\\
=&\frac{2u_{x}Z_{x}(u_{tx}+f'(u)u_{xx})
-(1+u_{x}^{2})\left(Z_{xt}+f'(u)Z_{xx}\right)}{Z^{2}_{x}}\nonumber\\
=&\frac{2u_{x}(u_{tx}+f'(u)u_{xx})
+f''(u)(1+u_{x}^{2})u_{x}}{Z_{x}}\nonumber\\
=&\frac{u_{x}}{Z_{x}}\left(2g(u)-2P+f''(u)+\rho^{2}\right)\nonumber\\
=&\left(g(u)-P+\frac{f''(u)}{2}+\frac{\rho^{2}}{2}\right)v\sin w.
\label{3016}
\end{align}

\section{Global solutions of the semi-linear system}

\setcounter{equation}{0}

\label{sec:4}

In this section, we will prove the global existence of solutions for the semi-linear system. Let initial data $(u_{0}(x), \rho_{0}(x))=(\bar{u}, \bar{\rho})\in H^{1}\times \left(L^{2}\cap L^{\infty}\right)$ be given. We can transfer problem \eqref{201}-\eqref{201a} into the following semi-linear system
\begin{align}
\begin{cases}
u_{T}=-P_{x},\\
\rho_{T}=-\left(\frac{1}{2}+\frac{f''(u)}{2}\right)\rho \tan \frac{w}{2},\\
w_{T}=2\left(g(u)-P+\frac{\rho^{2}}{2}\right)\cos^{2}\frac{w}{2}-f''(u)\sin^{2}\frac{w}{2},\\
v_{T}=\left(g(u)-P+\frac{f''(u)}{2}+\frac{\rho^{2}}{2}\right)v\sin w
\end{cases}
\label{401}
\end{align}
subject to the initial data
\begin{align}
\begin{cases}
u(0,Z)=\bar{u}(\bar{x}(Z)),\\
\rho(0,Z)=\bar{\rho}(\bar{x}(Z)),\\
w(0,Z)=2\arctan \bar{u}_{x}(\bar{x}(Z)),\\
v(0,Z)=1,
\end{cases}
\label{402}
\end{align}
where $P$ and $P_{x}$ are given by \eqref{3011}--\eqref{3012}, respectively.

It is easy to verify that system \eqref{401} is invariant under translation by $2\pi$ in $w$. For simplicity, we choose $w\in[-\pi,\pi]$. We now consider system \eqref{401} as an  ordinary differential equations in the Banach space
\begin{equation}
X:=H^{1}(\mathbb{R})\times\left(L^{2}(\mathbb{R})\cap L^{\infty}(\mathbb{R})\right)\times\left(L^{2}(\mathbb{R})\cap L^{\infty}(\mathbb{R})\right)\times L^{\infty}(\mathbb{R})
\label{403}
\end{equation}
with the norm
\begin{equation*}
\|(u,\rho,w,v)\|_{X}=\|u\|_{1}+\|\rho\|_{L^{2}}+\|\rho\|_{L^{\infty}}+\|w\|_{L^{2}}+\|w\|_{L^{\infty}}+\|v\|_{L^{\infty}}.
\end{equation*}

In what follows, we first prove the local existence of solutions for the Cauchy problem \eqref{401}--\eqref{402}. By the standard theory of ordinary differential equations in the Banach space, we only need to show that all functions on the right-hand side of system \eqref{401} are locally Lipschitz continuous. We then use the energy conservation property to extend the local solution to the global solution.

\begin{lemma}\label{lem401}
Let $(\bar{u},\bar{\rho})\in H^{1}\times\left(L^{2}\cap L^{\infty}\right)$. Then the Cauchy problem \eqref{401}--\eqref{402} has a unique solution defined on $[0,T]$ for some $T>0$.
\end{lemma}
\begin{proof}
To establish the local well-posedness, it suffices to prove the operator determined by the
right-hand side of \eqref{401}, which maps $(u,w,v)$ to
\begin{eqnarray}
\Big\{-P_{x},\ -\left(\frac{1}{2}+\frac{f''(u)}{2}\right)\rho \tan \frac{w}{2},\ 2\left(g(u)-P+\frac{\rho^{2}}{2}\right)\cos^{2}\frac{w}{2}-f''(u)\sin^{2}\frac{w}{2},\
\nonumber\\ \left(g(u)-P+\frac{f''(u)}{2}+\frac{\rho^{2}}{2}\right)v\sin w\Big\}
\label{404}
\end{eqnarray}
is Lipschitz continuous on every bounded domain $\Omega \subset X$ of the following form
\begin{eqnarray}
\Omega=\Big\{(u,\rho,w,v):\|u\|_{1}\leq \alpha, \|\rho\|_{L^{2}}\leq\beta, \|w\|_{L^{2}}\leq \theta, \|w\|_{L^{\infty}}\leq \pi, v(x)\in[v^{-},v^{+}]\nonumber\\
 \text{for a.e.}\ x\in \mathbb{R}\Big\}
\label{405}
\end{eqnarray}
for any positive constants $\alpha, \beta, \theta, v^{-}$ and $v^{+}$.

It is clear that the maps
\begin{align}
&-\left(\frac{1}{2}+\frac{f''(u)}{2}\right)\rho \tan \frac{w}{2},
\left(2g(u)+\rho^{2}\right))\cos^{2}\frac{w}{2},
-f''(u)\sin^{2}\frac{w}{2},\nonumber\\
&\left(g(u)+\frac{f''(u)}{2}+\frac{\rho^{2}}{2}\right)v\sin w,
\label{407}
\end{align}
are all Lipschitz continuous as maps from $\Omega$ into $L^{2}\cap L^{\infty}$.
Therefore, we only need to prove the maps
\begin{equation}
(u,\rho,w,v)\mapsto (P,P_{x})
\label{408}
\end{equation}
are Lipschitz continuous from $\Omega$ into $L^{2}\cap L^{\infty}$. To this end, it suffices to show that the above maps are Lipschitz continuous from $\Omega$ into $H^{1}$.

In what follows, we derive some estimates for future use. For $(u,w,v)\in\Omega$, we have
\begin{align}
measure\left\{Z\in \mathbb{R}:\left|\frac{w(Z)}{2}\right|\geq \frac{\pi}{4}\right\}
\leq& measure\left\{Z\in \mathbb{R}:\sin^{2}\frac{w(Z)}{2}\geq \frac{1}{4}\right\}\nonumber\\
\leq& 4\int_{\left\{Z\in \mathbb{R}:\sin^{2}\frac{w(Z)}{2}\geq \frac{1}{4}\right\}}\sin^{2}\frac{w(Z)}{2}dZ\nonumber\\
\leq& \int_{\left\{Z\in \mathbb{R}:\sin^{2}\frac{w}{2}\geq \frac{1}{4}\right\}}w^{2}(Z)dZ\nonumber\\
\leq& \mu^{2}.
\label{409}
\end{align}
Therefore, for any $Z_{1}<Z_{2}$, we get
\begin{equation}
\int_{Z_{1}}^{Z_{2}}v(s)\cos^{2}\frac{w(s)}{2}ds\geq \int_{\left\{s\in[Z_{1},Z_{2}],\left|\frac{w(s)}{2}\right|\leq \frac{\pi}{4}\right\}}\frac{v^{-}}{2}ds\geq\frac{v^{-}}{2}\left((Z_{2}-Z_{1})-\mu^{2}\right).
\label{4010}
\end{equation}
The inequality \eqref{4010} is a key estimate which guarantees that the exponential term in the formulate \eqref{3011}--\eqref{3012} for $P$ and $P_{x}$ decreasing quickly as $|Z-Z'|\rightarrow \infty$. Below, we introduce the exponentially decaying function
\begin{equation}
\Lambda(\eta):=\min \left\{1,e^{\left(\frac{\mu^{2}}{2}-\frac{|\eta|}{2}\right)v^{-}}\right\}.
\label{4011}
\end{equation}
Thus, we get
\begin{equation}
\|\Lambda(\eta)\|_{L^{1}}
=\left(\int_{|\eta|\leq \mu^{2}}+\int_{|\eta|\geq \mu^{2}}\right)\Lambda(\eta)d\eta
=2\mu^{2}+\frac{4}{v^{-}}.
\label{4012}
\end{equation}

In what follows, we prove that $P, P_{x}\in H^{1}$, namely,
\begin{equation}
P,\ \partial_{Z}P,\ P_{x},\ \partial_{Z}P_{x}\in L^{2}(\mathbb{R}).
\label{4013}
\end{equation}
It is obvious that the priori estimates for $P$ and $P_{x}$ are totally similar. For simplicity, we only consider the case for $P_{x}$.

From \eqref{3012}, we get
\begin{equation}
|P_{x}(Z)|\leq \frac{v^{+}}{2}\left|\Lambda\ast \left(g(u)\cos^{2} \frac{w}{2}+\frac{f''(u)}{2}\sin^{2} \frac{w}{2}
+\frac{\rho^{2}}{2}\cos^{2} \frac{w}{2}\right)(Z)\right|.
\label{4014}
\end{equation}
Therefore, using the standard properties of convolutions, the Sobolev inequality 
and Young's inequality, we get
\begin{align}
\|P_{x}(Z)\|_{L^{2}}&\leq \frac{v^{+}}{2}\|\Lambda\|_{L^{1}}\left(\|g(u)\|_{L^{2}}
+\frac{|f''(u)|}{2}\|w^{2}\|_{L^{2}}+\frac{1}{2}\|\rho^{2}\|_{L^{2}}\right)\nonumber\\
&\leq C\|\Lambda\|_{L^{1}}\left(\|u\|_{L^{2}}
+\frac{|f''(u)|}{2}\|w\|_{L^{\infty}}\|w\|_{L^{2}}+\|\rho\|_{L^{\infty}}\|\rho\|_{L^{2}}\right)\nonumber\\
&<\infty.
\label{4015}
\end{align}

Next, differentiating $P_{x}$ with respect to $Z$, we get
\begin{align}
\partial_{Z}P_{x}(Z)=&-\left(g(u(Z))\cos^{2} \frac{w(Z)}{2}+\frac{f''(u(Z))}{2}\sin^{2} \frac{w(Z)}{2}+\frac{\rho^{2}}{2}\cos^{2} \frac{w}{2}\right)v(Z)\nonumber\\
&+\frac{1}{2}\left(\int_{Z}^{+\infty}
-\int_{-\infty}^{Z}\right)e^{-\left|\int_{Z}^{Z'}(v\cos^{2} \frac{w}{2})(T,s)ds\right|}v(Z)\cos^{2} \frac{w(Z)}{2}\sign{Z'-Z}\nonumber\\
&\cdot \left(g(u(Z'))\cos^{2} \frac{w(Z')}{2}+\frac{f''(u(Z'))}{2}\sin^{2} \frac{w(Z')}{2}+\frac{\rho^{2}}{2}\cos^{2} \frac{w}{2}\right)v(Z')dZ'.
\label{4016}
\end{align}
Therefore,
\begin{align}
|\partial_{Z}P_{x}(Z)|\leq&v^{+}\left|g(u(Z))+\frac{|f''(u(Z))|}{8}w^{2}+\frac{\rho^{2}}{2}\right|\nonumber\\
&+\frac{(v^{+})^{2}}{2}\left|\Lambda\ast \left(g(u)\cos^{2} \frac{w}{2}+\frac{f''(u)}{2}\sin^{2} \frac{w}{2}+\frac{\rho^{2}}{2}\cos^{2} \frac{w}{2}\right)(Z)\right|.
\label{4017}
\end{align}
Furthermore, applying standard properties of convolutions and Young's inequality, we get
\begin{align}
\|\partial_{Z}P_{x}(Z)\|_{L^{2}}\leq& v^{+}\left(\|g(u)\|_{L^{2}}+\frac{|f''(u)|}{8}\|w^{2}\|_{L^{2}}+\frac{1}{2}\|\rho^{2}\|_{L^{2}}\right)\nonumber\\
&+\frac{(v^{+})^{2}}{2}\|\Lambda\|_{L^{1}} \left(\|g(u)\|_{L^{2}}+\frac{|f''(u)|}{8}\|w^{2}\|_{L^{2}}+\frac{1}{2}\|\rho^{2}\|_{L^{2}}\right)\nonumber\\
\leq& C\left(v^{+}+\frac{(v^{+})^{2}}{2}\|\Lambda\|_{L^{1}}\right)\left(
\|u\|_{L^{2}}+\frac{|f''(u)|}{8}\|w\|_{L^{\infty}}\|w\|_{L^{2}}
+\frac{1}{2}\|\rho\|_{L^{\infty}}\|\rho\|_{L^{2}}\right)\nonumber\\
<& \infty.
\label{4018}
\end{align}
Thus, $P_{x}\in H^{1}(\mathbb{R})$. Note that the estimates for $P$ and $P_{x}$ can be obtained by the same method. Therefore, the proof of the relation \eqref{4013} is completed.

In what follows, we verify the Lipschitz continuity of the map given in \eqref{409}. This can be done by proving that for $(u,\rho,w,v)\in \Omega$, the partial derivatives
\begin{equation}
\frac{\partial P}{\partial u},\ \frac{\partial P}{\partial \rho},\ \frac{\partial P}{\partial w},\ \frac{\partial P}{\partial v},\
\frac{\partial P_{x}}{\partial u},\ \frac{\partial P_{x}}{\partial \rho},\ \frac{\partial P_{x}}{\partial w},\ \frac{\partial P_{x}}{\partial v}
\label{4019}
\end{equation}
are uniformly bounded linear operators from the appropriate spaces into $H^{1}$. Due to the fact that all the partial derivatives can be estimated by the same method, so we only detail the argument for $\frac{\partial P_{x}}{\partial u}$.

For every test function $\phi\in H^{1}$, the operators $\frac{\partial P_{x}}{\partial u}$ and $\frac{\partial (\partial_{Z}P_{x})}{\partial u}$ at a given point $(u,w,v)\in \Omega$ are defined by
\begin{align}
&\left(\frac{\partial P_{x}(u,\rho,w,v)}{\partial u}\cdot \phi\right)(Z)\nonumber\\
=&\frac{1}{2}\left(\int_{Z}^{+\infty}-\int_{-\infty}^{Z}\right)
e^{-\left|\int_{Z}^{Z'}(v\cos^{2} \frac{w}{2})(T,s)ds\right|}\left(g'(u(Z'))\cos^{2} \frac{w(Z')}{2}+\frac{f'''(u(Z'))}{2}\sin^{2} \frac{w(Z')}{2}\right)\nonumber\\
&\cdot v(Z')\phi(Z') dZ'
\label{4020}
\end{align}
and
\begin{align}
&\left(\frac{\partial (\partial_{Z}P_{x})(u,\rho,w,v)}{\partial u}\cdot \phi\right)(Z)\nonumber\\
=&-\left(g'(u(Z))\cos^{2} \frac{w}{2}+\frac{f'''(u(Z))}{2}\sin^{2} \frac{w}{2}\right)v(Z)\phi(Z)\nonumber\\
&+\frac{1}{2}\left(\int_{Z}^{+\infty}-\int_{-\infty}^{Z}\right)
e^{-\left|\int_{Z}^{Z'}(v\cos^{2} \frac{w}{2}ds)(T,s)\right|}v(Z)\cos^{2} \frac{w(Z)}{2}\sign{Z'-Z}\nonumber\\
&\cdot \left(g'(u(Z'))\cos^{2} \frac{w(Z')}{2}+\frac{f'''(u(Z'))}{2}\sin^{2} \frac{w(Z')}{2}\right)v(Z')\phi(Z') dZ'.
\label{4021}
\end{align}
Therefore, we obtain
\begin{align}
\left\|\frac{\partial P_{x}}{\partial u}\cdot \phi\right\|_{L^{2}}
\leq&v^{+}\left\|\Lambda\ast \left(g'(u)+\frac{f'''(u)}{2}\right)\right\|_{L^{2}}\|\phi\|_{L^{\infty}}\nonumber\\
\leq& Cv^{+}\|\Lambda\|_{L^{1}} \|u\|_{L^{2}}\|\phi\|_{1}\nonumber\\
<& \infty
\label{4022}
\end{align}
and
\begin{align}
&\left\|\frac{\partial(\partial_{Z} P_{x})}{\partial u}\cdot \phi\right\|_{L^{2}}\nonumber\\
\leq& v^{+}\left\|g'(u)+\frac{f'''(u)}{2}\right\|_{L^{2}}\|\phi\|_{L^{\infty}}
+\frac{(v^{+})^{2}}{2}\left\|\Lambda\ast \left(g'(u)+\frac{f'''(u)}{2}\right)\right\|_{L^{2}}\|\phi\|_{L^{\infty}}\nonumber\\
\leq&\left(v^{+}+\frac{(v^{+})^{2}}{2}\|\Lambda\|_{L^{1}}\right) \left(\|g'(u)+\frac{f'''(u)}{2}\|_{L^{2}}\right)\|\phi\|_{1}\nonumber\\
<& \infty,
\label{4023}
\end{align}
where we used the facts that
$$\|u\|_{L^{\infty}}\leq \|u\|_{1},\ \
\|\phi\|_{L^{\infty}}\leq \|\phi\|_{1}
\ \text{and}\ |g'(u)|\leq C(\|u\|_{1})|u|.$$
Hence we obtain that $\frac{\partial P_{x}}{\partial u}$ is a bounded linear operator from $H^{1}$ into $H^{1}$. As above, we can obtain the boundedness of other partial derivatives, thus the uniform Lipschitz continuous of the map in \eqref{408} is verified. Then using the standard ODE theory in the Banach space, we can establish the local existence of solutions for the Cauchy problem \eqref{401}--\eqref{402}, namely, the Cauchy problem \eqref{401}--\eqref{402} has a unique solution on $[0,T]$ for some $T>0$.
This completes the proof of Lemma \ref{lem401}.
\end{proof}

In what follows, we extend the local solution obtained in Lemma \ref{lem401} globally. To this end, it suffices to prove that for all $T<\infty$,
\begin{equation}
\|u\|_{1}+\|\rho\|_{L^{2}}+\|\rho\|_{L^{\infty}}+\|w\|_{L^{2}}+\|w\|_{L^{\infty}}
+\|v\|_{L^{\infty}}+\left\|\frac{1}{v}\right\|_{L^{\infty}}<\infty.
\label{4024}
\end{equation}

\begin{lemma}\label{lem402}
Let $(\bar{u},\bar{\rho})\in H^{1}\times\left(L^{2}\cap L^{\infty}\right)$. Then the Cauchy problem \eqref{401}--\eqref{402} has a unique solution defined for all $T>0$.
\end{lemma}
\begin{proof}
For the local solution obtained in Lemma \ref{lem401}, we claim that
\begin{equation}
u_{Z}=\frac{u_{x}}{Z_{x}}
=\frac{u_{x}}{1+u_{x}^{2}}v
=\frac{1}{2}v\sin w.
\label{4025}
\end{equation}
Indeed, from \eqref{401}, we have
\begin{align}
u_{ZT}=u_{TZ}=&-\partial_{Z}P_{x}\nonumber\\
=&\left(\left(g(u)-P(Z)+\frac{\rho^{2}}{2}\right)\cos^{2} \frac{w}{2}+\frac{f''(u)}{2}\sin^{2}\frac{w}{2}\right)v(Z).
\label{4026}
\end{align}
On the other hand,
\begin{align}
&\left(\frac{1}{2}v\sin w\right)_{T}\nonumber\\
=&\frac{1}{2}v_{T}\sin w+\frac{1}{2}vw_{T}\cos w\nonumber\\
=&\frac{1}{2}v\sin^{2} w\left(g(u)-P(Z)+\frac{f''(u)}{2}+\frac{\rho^{2}}{2}\right)
+\frac{v}{2}\cos w\Big(2\left(g(u)-P(Z)+\frac{\rho^{2}}{2}\right)\cos^{2} \frac{w}{2}
\nonumber\\
&-f''(u)\sin^{2} \frac{w}{2}\Big)\nonumber\\
=&\left(\left(g(u)-P(Z)+\frac{\rho^{2}}{2}\right)\cos^{2} \frac{w}{2}+\frac{f''(u)}{2}\sin^{2}\frac{w}{2}\right)v(Z).
\label{4027}
\end{align}
Applying the initial data, we know
\begin{equation}
u_{Z}=\frac{1}{2}\sin w\ \ \text{and}
\ \ v=1,\ \ as\ \ T=0,
\label{4028}
\end{equation}
which means that \eqref{4025} holds initially. Thus, we infer that \eqref{4025} remains valid for all $T$ as long as the solution exists.

In what follows, we verify the boundedness of \eqref{4024}. To this end, we check the conservation law $E(t)$. In the new system \eqref{401}--\eqref{402}, the conservation law of $E(T)$ read
\begin{equation}
E(T)=\int_{\mathbb{R}}\left(u^{2}\cos^{2}\frac{w}{2}
+\sin^{2}\frac{w}{2}+\rho^{2}\cos^{2}\frac{w}{2}\right)v(T,Z)dZ=\bar{E}(0).
\label{4029}
\end{equation}
To prove \eqref{4029}, it is useful to give the following identities in terms of the $Z$-derivatives.
\begin{equation}
P_{Z}=v(Z) P_{x}(Z)\cos^{2} \frac{w(Z)}{2}
\label{4030}
\end{equation}
and
\begin{equation}
\partial_{Z}P_{x}=-\left((g(u)-P(Z))\cos^{2} \frac{w}{2}+\frac{f''(u)}{2}\sin^{2}\frac{w}{2}+\frac{\rho^{2}}{2}\cos^{2}\frac{w}{2}\right)v(Z).
\label{4031}
\end{equation}

Applying \eqref{401}, \eqref{4030} and \eqref{4031}, a direct calculation reveals that
\begin{align}
\frac{dE(T)}{dT}=&\int_{\mathbb{R}}\left(\left(u^{2}\cos^{2}\frac{w}{2}
+\sin^{2}\frac{w}{2}+\rho^{2}\cos^{2}\frac{w}{2}\right)v(T,Z)\right)_{T}dZ\nonumber\\
=&\int_{\mathbb{R}}\bigg(\Big(2uu_{T}\cos^{2}\frac{w}{2}
-u^{2}w_{T}\cos\frac{w}{2}\sin\frac{w}{2}
+w_{T}\sin \frac{w}{2}\cos \frac{w}{2}\nonumber\\
&+2\rho\rho_{T}\cos^{2}\frac{w}{2}
-\rho^{2}w_{T}\cos\frac{w}{2}\sin\frac{w}{2}\Big)v
+\left(u^{2}\cos^{2}\frac{w}{2}+\sin^{2}\frac{w}{2}
+\rho^{2}\cos^{2}\frac{w}{2}\right)v_{T}\bigg)dZ\nonumber\\
=&\int_{\mathbb{R}}\left(-2uP_{x}\cos^{2}\frac{w}{2}+\frac{f''(u)}{2}u^{2}\sin w
+(g(u)-P(Z))\sin w\right)v(T,Z)dZ.
\label{4032}
\end{align}
In view of \eqref{4025} and \eqref{4030}, we have
\begin{equation}
(uP)_{Z}=u_{Z}P+uP_{Z}=vP\sin \frac{w}{2}\cos \frac{w}{2}+uvP_{x}\cos^{2} \frac{w}{2}
\label{4033}
\end{equation}
and
\begin{equation}
g(u)v\sin w=g(u)u_{Z}=(G(u))_{Z},
\label{4034}
\end{equation}
where $G(u)=\int_{0}^{u}g(s)ds$.

On the other hand,
\begin{align}
\frac{f''(u)}{2}vu^{2}\sin w
=&f''(u)u^{2}u_{Z}\nonumber\\
=&\left(f'(u)u^{2}\right)_{Z}-2f'(u)uu_{Z}\nonumber\\
=&\left(f'(u)u^{2}\right)_{Z}-2\left((f(u)u)_{Z}-f(u)u_{Z}\right)\nonumber\\
=&\left(f'(u)u^{2}\right)_{Z}-2(f(u)u)_{Z}+2f(u)u_{Z}\nonumber\\
=&\left(f'(u)u^{2}\right)_{Z}-2(f(u)u)_{Z}+2(F(u))_{Z},
\label{4035}
\end{align}
where $F(u)=\int_{0}^{u}f(s)ds$.

Therefore,
\begin{align}
\frac{dE(T)}{dT}=&\int_{\mathbb{R}}\left(\left(u^{2}\cos^{2}\frac{w}{2}
+\sin^{2}\frac{w}{2}+\rho^{2}\cos^{2}\frac{w}{2}\right)v(T,Z)\right)_{T}dZ\nonumber\\
=&\int_{\mathbb{R}}\left(G(u)-2uP+f'(u)u^{2}-2f(u)u+2F(u)\right)_{Z}dZ\nonumber\\
=&0,
\label{4036}
\end{align}
where in deriving the last equality we have used the asymptotic property
\begin{equation*}
\lim_{|Z|\rightarrow \infty}u(Z)=0\ \ \text{as}\ \ u\in H^{1}(\mathbb{R}),
\end{equation*}
and the fact that $P(Z)$ is uniformly bounded. This proves \eqref{4029}.

We have now proved the conservation law \eqref{4029} in the new variables along any solution of \eqref{401}--\eqref{402}. In what follows, we use the conservation law \eqref{4029} to derive a priori estimate on $\|u(T)\|_{L^{\infty}}$. It is clear that
\begin{align}
\sup_{Z\in\mathbb{R}}\left|u^{2}(T,Z)\right|
\leq 2\int_{\mathbb{R}}\left|uu_{Z}\right|dZ
&\leq 2\int_{\mathbb{R}}\left|u\sin \frac{w}{2} \cos \frac{w}{2}\right|v dZ\nonumber\\
&\leq \int_{\mathbb{R}} \left|\sin^{2} \frac{w}{2}+u^{2}\cos^{2} \frac{w}{2}\right|vdZ\nonumber\\
&\leq \bar{E}(0).
\label{4037}
\end{align}

From \eqref{4029} and \eqref{3011}, we can easily verify
\begin{align}
\|P(T)\|_{L^{\infty}}\leq& \frac{1}{2}\left\|e^{-|x|}\right\|_{L^{\infty}}
\left\|g(u)+\frac{f''(u)}{2}u_{x}^{2}+\frac{\rho^{2}}{2}\right\|_{L^{1}}
\leq C\bar{E}(0).
\label{4038}
\end{align}
Similarly, we can obtain
\begin{equation*}
\|P_{x}(T)\|_{L^{\infty}}\leq C\bar{E}(0).
\end{equation*}
Hence we recover the estimate \eqref{2010} in the new variables.

Below, using the estimates \eqref{4037}, \eqref{4038} and the third equation in system \eqref{401}, we can prove the  $L^{\infty}$ bound for $v(T,Z)$. Indeed, we have
\begin{equation}
|v_{T}(T,Z)|\leq C\bar{E}(0)v(T,Z).
\label{4039}
\end{equation}
Since $v(0,Z)=1$, \eqref{4039} yields
\begin{equation}
e^{-C\bar{E}(0)T}\leq v(T,Z)\leq e^{C\bar{E}(0)T}.
\label{4040}
\end{equation}
Similarly, it follows from the second equation of system \eqref{401} that
\begin{equation}
|w_{T}(T,Z)|\leq C,
\label{4041}
\end{equation}
where $C=C(\bar{E}(0))>0$. Consequently,
\begin{equation}
\|w(T,Z)\|_{L^{\infty}}\leq \|w(0,Z)\|_{L^{\infty}}+CT.
\label{4042}
\end{equation}

In what follows, we prove that $\|u\|_{1}$ is bounded for any bounded internal of time $T$. To this end, multiplying $2u$ to the first equation of system \eqref{401}, we get
\begin{equation}
\frac{d}{dT}\|u(T)\|_{L^{2}}^{2}\leq 2\|u(T)\|_{L^{\infty}}\|P_{x}(T)\|_{L^{1}}.
\label{4043}
\end{equation}
Differentiating the first equation of system \eqref{401} with respect to $Z$, we get
\begin{equation}
u_{TZ}(T,Z)=-\partial_{Z}P_{x}(T,Z).
\label{4044}
\end{equation}
Multiplying \eqref{4044} by $2u_{Z}$, we obtain
\begin{equation}
\frac{d}{dT}\|u_{Z}(T)\|_{L^{2}}^{2}\leq 2\|u_{Z}(T)\|_{L^{\infty}}\|\partial_{Z}P_{x}(T)\|_{L^{1}}.
\label{4045}
\end{equation}

On the other hand, it is known from \eqref{4025} that
\begin{equation}
\|u_{Z}(T)\|_{L^{\infty}}\leq \frac{1}{2}\|v(T)\|_{L^{\infty}}\leq \frac{1}{2}e^{C\bar{E}(0)T}.
\label{4046}
\end{equation}
In order to prove that $\|u\|_{1}$ is bounded for any $T<\infty$, it suffices to show $\|P_{x}(T)\|_{L^{1}}$ and $\|\partial_{Z}P_{x}(T)\|_{L^{1}}$ are bounded. It is observed that these two terms can be estimated by the similar method, we only need to consider
$\|\partial_{Z}P_{x}(T)\|_{L^{1}}$.

Indeed, for $Z<Z'$, we have
\begin{align}
\int^{Z'}_{Z}(v\cos^{2} \frac{w}{2})(T,s)ds
\geq&\int_{\big\{s\in[Z,Z'],\left|\frac{w}{2}\right|\leq\frac{\pi}{4}\big\}}(v\cos^{2} \frac{w}{2})(s)ds\nonumber\\
\geq&\int_{\big\{s\in[Z,Z'],\left|\frac{w}{2}\right|\leq\frac{\pi}{4}\big\}}\frac{v(s)}{2}ds\nonumber\\
\geq&\frac{v^{-}}{2}(Z'-Z)
-\int_{\big\{s\in[Z,Z'],\left|\frac{w}{2}\right|\geq\frac{\pi}{4}\big\}}\frac{v(s)}{2}ds\nonumber\\
\geq&\frac{v^{-}}{2}(Z'-Z)
-\int_{\big\{s\in[Z,Z'],\left|\frac{w}{2}\right|\geq\frac{\pi}{4}\big\}}v(s)\sin^{2}\frac{w(s)}{2}ds
\nonumber\\
\geq&\frac{v^{-}}{2}(Z'-Z)-\bar{E}(0).
\label{4047}
\end{align}
where $v^{-}=e^{-C\bar{E}(0)T}$.

Below, we introduce the exponentially decaying function
\begin{equation}
\Gamma(\eta):=\min\left\{1,e^{\bar{E}(0)-\frac{v^{-}|\eta|}{2}}\right\}
\label{4048}
\end{equation}
with
\begin{equation}
\|\Gamma(\eta)\|_{L^{1}}=\frac{4(\bar{E}(0)+1)}{v^{-}}=4e^{C\bar{E}(0)T}(\bar{E}(0)+1).
\label{4049}
\end{equation}
Therefore, from \eqref{4031} and \eqref{4049}, we get
\begin{align}
\|\partial_{Z}P_{x}\|_{L^{1}}
=&\left\|-\left(g(u)\cos^{2}\frac{w}{2}+\frac{f''(u)}{2}\sin^{2}\frac{w}{2}
+\frac{\rho^{2}}{2}\cos^{2}\frac{w}{2}\right)v
+vP\cos^{2}\frac{w}{2}\right\|_{L^{1}}\nonumber\\
\leq&C\bar{E}(0)+\frac{v^{+}}{2}\left\|\Gamma\ast
\left(g(u)\cos^{2}\frac{w}{2}+\frac{f''(u)}{2}\sin^{2}\frac{w}{2}+\frac{\rho^{2}}{2}\cos^{2}\frac{w}{2}\right)v\right\|_{L^{1}}\nonumber\\
\leq&C\bar{E}(0)+\frac{v^{+}}{2}\|\Gamma\|_{L^{1}}
\left\|\left(g(u)\cos^{2}\frac{w}{2}+\frac{f''(u)}{2}\sin^{2}\frac{w}{2}+\frac{\rho^{2}}{2}\cos^{2}\frac{w}{2}\right)v\right\|_{L^{\infty}}\nonumber\\
\leq&C\bar{E}(0)+Cv^{+}e^{C\bar{E}(0)T}(\bar{E}(0)+1)\bar{E}(0)\nonumber\\
<&\infty.
\label{4050}
\end{align}
It then turns out that $\|u\|_{1}$ is bounded on the bounded internals of time $T$.

From the second equation in \eqref{401}, we can obtain that
\begin{equation*}
\|\rho\|_{L^{\infty}}\leq e^{KT},
\end{equation*}
and
\begin{equation*}
\frac{d}{dt}\|\rho\|_{L^{2}}^{2}\leq 2K \|\rho\|_{L^{2}}^{2},
\end{equation*}
for a suitable constant $K=K(E(0))$, which implies that $\|\rho\|_{L^{\infty}}$ and $\|\rho\|_{L^{2}}$ remain bounded on bounded intervals of time.

Finally, multiplying $2w$ to the third equation of system \eqref{401}, we get
\begin{align}
\frac{d}{dT}\|w(T)\|_{L^{2}}^{2}
\leq& 2\int_{\mathbb{R}}|(2g(u)-2P+\rho^{2})w|dz+\frac{|f''(u)|}{2}\int_{\mathbb{R}}|w^{3}|dZ\nonumber\\
\leq& C(\|u\|_{L^{2}}+\|P\|_{L^{2}}+\|\rho\|_{L^{2}}^{2})
\|w\|_{L^{\infty}}+\frac{|f''(u)|}{2}\|w\|_{L^{\infty}}\|w\|_{L^{2}}^{2}
.
\label{4051}
\end{align}
By the previous bounds, it is clear that $\|w\|_{L^{2}}$ remains bounded on bounded internals of time $T$. This completes the proof that the local solution of system \eqref{401} can be extended globally in time.
\end{proof}

Furthermore, similar to the result in \cite{Bressan2007}, we have the following property for the global solution in Lemma \eqref{lem402}.
\begin{lemma}\label{lem403}
Consider the set of time
\begin{equation*}
\Theta:=\{T\geq 0,\ measure\{Z\in \mathbb{R}:w(T,Z)=-\pi\}>0\}.
\end{equation*}
Then
\begin{equation}
measure(\Theta)=0.
\label{4052}
\end{equation}
\end{lemma}

\section{Solutions to the nonlinear dispersive wave equations}
\setcounter{equation}{0}

\label{sec:5}

In this section, we construct the weak solution to \eqref{201}-\eqref{201a} by using an inverse translation on the solution of system \eqref{401}.

We define $t$ and $x$ as functions of $T$ and $Z$ by
\begin{equation}
x(T,Z)=\bar{x}(Z)+\int_{0}^{T}f'(u(\zeta,Z))d\zeta,\ \ t=T.
\label{501}
\end{equation}
Thus the above function $x(T,Z)$ provides a solution to the following initial problem
\begin{equation}
\frac{\partial x(T,Z)}{\partial T}=f'(u(T,Z)),\ \ x(0,Z)=\bar{x}(Z),
\label{502}
\end{equation}
which means that $x(T,Z)$ is a characteristic.

In what follows, we will prove that the function
\begin{equation}
u(t,x)=u(T,Z),\ \ \text{if}\ \ t=T,\ x=x(T,Z),
\label{503}
\end{equation}
provides a weak solution of \eqref{201}-\eqref{201a}.

\begin{definition}\label{def501}
By a solution of the Cauchy problem \eqref{201}--\eqref{201} on $[t_{1},t_{2}]$ we mean a pair of H\"older continuous $(u(t,x),\rho(t,x))$ defined on $[t_{1},t_{2}]\times \mathbb{R}$ with the following properties. At each fixed $t$, we have $(u(t,\cdot),\rho(t,\cdot))\in H^{1}\times (L^{2}\cap L^{\infty})$. Moreover, the maps
$t\mapsto (u(t,\cdot),\rho(t,\cdot))$ are Lipschitz continuous from $[t_{1},t_{2}]$ into $L^{2}(\mathbb{R})$, and satisfy the initial data \eqref{103} together with
\begin{equation}
\begin{cases}
u_{t}=-f'(u)u_{x}-P_{x},\\
\rho_{t}=-f'(u)\rho_{x}-\left(\frac{1}{2}+\frac{f''}{2}\right)u_{x}\rho,
\end{cases}
\label{504}
\end{equation}
for a.e. $t$. Here \eqref{504} is understood as an equatlity between functions in $L^{2}(\mathbb{R})$.
\end{definition}

In what follows, we state the main result on the global well-posedness of the energy conservative solution for \eqref{201}-\eqref{201a}.

\begin{theorem}\label{the501}
Let the initial data $(\bar{u}(x),\bar{\rho}(x))\in H^{1}\times (L^{2}\cap L^{\infty})$. Then problem \eqref{201}-\eqref{201a} with the initial data $(\bar{u}(x),\bar{\rho}(x))$ has a global energy conservative solution $(u(t,x),\rho(t,x))$ in the sense of Definition \ref{def501}. Furthermore, the solution $(u(t,x),\rho(t,x))$ satisfy the following properties:\\
(i) $(u(t,x),\rho(t,x))$ are uniformly H\"older continuous with exponent $\frac{1}{2}$ on both $t$ and $x$.\\
(ii) The energy $u^{2}+u_{x}^{2}+\rho^{2}$ is almost conserved, i.e.,
\begin{equation}
\|u\|_{1}^{2}+\|\rho\|_{L^{2}}^{2}=\|\bar{u}\|_{1}^{2}+\|\bar{\rho}\|_{L^{2}}^{2},\ \ \text{for\ a.e.}\ \ t\in \mathbb{R}_{+}.
\label{505}
\end{equation}
(iii) The solution $(u(t,x),\rho(t,x))$ are continuously depending on the initial datum $(\bar{u}(x),\bar{\rho}(x))$. That is, let $(\bar{u}_{n}(x),\bar{\rho}_{n}(x))$ be a sequence of initial data such that
\begin{equation*}
\|\bar{u}_{n}-\bar{u}\|_{1}\rightarrow 0,
\|\bar{\rho}_{n}-\bar{\rho}\|_{L^{2}}\rightarrow 0,
\|\bar{\rho}_{n}-\bar{\rho}\|_{L^{\infty}}\rightarrow 0,
\ \ \text{as}\ \ n\rightarrow \infty.
\end{equation*}
Then the corresponding solutions $(u_{n}(t,x),\rho_{n}(t,x))$ converge to $(u(t,x),\rho(t,x))$ uniformly for $(t,x)$ in any bounded sets.
\end{theorem}

\begin{proof}
The argument of proof is divided into five steps.

\textbf{Step 1.} We show that the continuous map $(T,Z)\rightarrow (t,x(T,Z))$
is a surjective function in $\mathbb{R}^{2}$. Indeed, by \eqref{403} and \eqref{501}, we get
\begin{equation*}
\bar{x}(Z)-\sqrt{E(0)}T\leq x(T,Z)\leq \bar{x}(Z)+\sqrt{E(0)}T.
\end{equation*}
Then from \eqref{303}, we deduce that
\begin{equation*}
\lim_{Z\rightarrow \pm\infty}x(T,Z)=\pm\infty.
\end{equation*}
Therefore, the image of continuous map $(T,Z)\rightarrow (t,x(T,Z))$ covers the entire half-plane $\mathbb{R}^{+}\times \mathbb{R}$.

\textbf{Step 2.} We claim that
\begin{equation}
x_{Z}=v\cos^{2} \frac{w}{2}\ \ \text{for}\ \ T\geq 0,\ \ \text{a.e.}\ \ Z\in \mathbb{R}.
\label{505}
\end{equation}
In fact, from \eqref{401} and \eqref{4025}, we get
\begin{align}
\left(v\cos^{2} \frac{w}{2}\right)_{T}
=&-w_{T}\sin \frac{w}{2}\cos \frac{w}{2}+v_{T}\cos^{2} \frac{w}{2}\nonumber\\
=&-v\sin \frac{w}{2}\cos \frac{w}{2}\left(2(g(u)-P(Z))\cos^{2} \frac{w}{2}
-f''(u)\sin^{2} \frac{w}{2}\right)\nonumber\\
&+v\sin w\cos^{2} \frac{w}{2}\left(g(u)-P(Z)+\frac{f''(u)}{2}\right)\nonumber\\
=&\frac{f''(u)}{2}v\sin w\nonumber\\
=&\left(f'(u)\right)_{Z}.
\label{506}
\end{align}

On the other hand, \eqref{502} implies
\begin{equation}
\frac{\partial x_{Z}}{\partial T}=\left(f'(u)\right)_{Z}.
\label{507}
\end{equation}
Since the function $x\rightarrow 2\arctan \bar{u}_{x}(x)$ is measure, the identity \eqref{505} holds for almost every $Z\in \mathbb{R}$ and $T=0$. By the above computations, it remains true for all times $T\geq 0$. Furthermore, the function $x(T,Z)$ is non-decreasing on $Z$ when $T$ is fixed.

\textbf{Step 3.} Our goal is to show that $u(t,x)=u(T,x(T,Z))$ is well defined. In fact,
if $x(T,Z_{1})=x(T,Z_{2})$ for $Z_{1}<Z_{2}$, then we have
\begin{equation*}
x(T,Z)=x(T,Z_{1})\ \text{for}\ Z\in[Z_{1},Z_{2}],
\end{equation*}
where we use the non-decreasing property of $x(T,Z)$ on $Z$. From \eqref{505}, we get
\begin{equation*}
\cos \frac{w(T,Z)}{2}=0\ \text{for}\ Z\in[Z_{1},Z_{2}].
\end{equation*}
Therefore,
\begin{equation*}
u(T,Z_{2})-u(T,Z_{1})=\int_{Z_{1}}^{Z_{2}}\frac{v}{2}\sin wds=0.
\end{equation*}
This proves $u(t,x)\rightarrow u(T,x(T,Z))$ is well defined for all $t\geq 0$ and $x\in \mathbb{R}$.

\textbf{Step 4.} We discuss the regularity of $(u(t,\cdot),\rho(t,\cdot))$ and energy conservation, and prove that the map $t\mapsto (u(t,\cdot),\rho(t,\cdot))$ are Lipschitz continuous on $L^{2}(\mathbb{R})$. By \eqref{4012}, we know that $E(T)$ is conservative on $(T,Z)$ coordinates.
For any given time $t$, we have
\begin{align}
E(T)
=&\int_{\mathbb{R}}\left(u^{2}\cos^{2}\frac{w}{2}+\sin^{2}\frac{w}{2}+\rho^{2}\cos^{2}\frac{w}{2}\right)vdZ\nonumber\\
=&\int_{\{\cos w>-1\}}\left(u^{2}\cos^{2}\frac{w}{2}+\sin^{2}\frac{w}{2}+\rho^{2}\cos^{2}\frac{w}{2}\right)vdZ\nonumber\\
=&E(0).
\label{508}
\end{align}
By Lemma \ref{lem403}, we obtain that \eqref{504} holds for almost all $t$.

Applying a Sobolev inequality in \cite{Bressan2007}, we obtain the uniform H\"older continuity with the exponent $\frac{1}{2}$ for $(u(t,x),\rho(t,x))$ as functions of $x$. Besides, from the first equation in \eqref{401} and the uniform bounds on $\|P_{x}\|_{L^{\infty}}$, we can deduce the map $t\mapsto u(t,\cdot)$ is Lipschitz continuous on $L^{\mathbb{R}}$. Similar calculation shows the Lipschitz continuity of $\rho(t,\cdot)$ as a map $t\mapsto \rho(t,\cdot).$

\textbf{Step 5.} The proof of global energy conservative solution $(u(t,x),\rho(t,x))$ in the sense of Definition \ref{def501}, and the continuous dependence result can be directly obtained by following the argument in \cite{Bressan2007}.

\end{proof}

{\bf Acknowledgement.}
This work is partially supported by NSFC Grant (no. 12201539) and Natural Science Foundation of Xinjiang Uygur Autonomous Region (no. 2022D01C65) and Natural Science Foundation of Gansu Province (no. 23JRRG0006) and The Youth Doctoral Support Project for Universities in Gansu Province (no. 2024QB-106).
\bibliographystyle{elsarticle-num}
\bibliography{<your-bib-database>}



\end{document}